\documentclass[10pt]{article}
\usepackage[utf8]{inputenc} 
\usepackage{amsmath, amssymb, amsthm,graphicx}
\usepackage{natbib}
\newtheorem{rema}{Remark}

\newtheorem{lemma}{Lemma}

\newtheorem{prop}{Proposition}
\newtheorem{thm}{Theorem}
\newtheorem{defin}{Definition}

\def \n{\Vert}

\newcommand{\E}{\mathop{\mathbb{E}}}
\newcommand{\dist}{\mathop{\mathrm{dist}}}
\renewcommand{\P}{\mathop{\mathbb{P}}}

\def\R{{\mathbb{R}}}
\def\N{{\mathbb{N}}}

\def\F{{\cal{F}}}

\def\|{\,|\,}
\def\bn#1\en{\begin{align*}#1\end{align*}}
\def\bnn#1\enn{\begin{align}#1\end{align}}

\title{CONSTRUCTIVE AND CONSISTENT ESTIMATION OF QUADRATIC MINIMAX}

\begin{document} 
\date{}
\author{Philip Kennerberg and Ernst C. Wit}
\maketitle


\begin{abstract}
We consider $k$ square integrable random variables $Y_1,...,Y_k$ and $k$ random (row) vectors of length $p$, $X_1,...,X_k$ such that $X_i(l)$ is square integrable for $1\le i\le k$ and $1\le l\le p$. No assumptions whatsoever are made of any relationship between the $X_i$:s and $Y_i$:s. We shall refer to each pairing of $X_i$ and $Y_i$ as an environment. We form the square risk functions $R_i(\beta)=\E\left[(Y_i-\beta X_i)^2\right]$ for every environment and consider $m$ affine combinations of these $k$ risk functions. Next, we define a parameter space $\Theta$ where we associate each point with a subset of the unique elements of the covariance matrix of $(X_i,Y_i)$ for an environment. Then we study estimation of the $\arg\min$-solution set of the maximum of a the $m$ affine combinations the of quadratic risk functions. We provide a constructive method for estimating the entire $\arg\min$-solution set which is consistent almost surely outside a zero set in $\Theta^k$. This method is computationally expensive, since it involves solving polynomials of general degree. To overcome this, we define another approximate estimator that also provides a consistent estimation of the solution set based on the bisection method, which is computationally much more efficient. We apply the method to worst risk minimization in the setting of structural equation models. 
\end{abstract}

\section{Introduction}
Minimax estimation is used in statistical decision theory to evaluate the performance of estimators in the worst-case scenario. The main idea is to find estimators that minimize the maximum possible risk or loss over all possible parameter values in a given set. Minimax estimation is particularly important in contexts where there is uncertainty about the true parameter value or even the true underlying model, and where the consequences of estimation errors are severe. 

Minimax estimation is used in various contexts, such as robust estimation \citep{huber1981robust,hampel2011robust}. When dealing with data that may be contaminated by outliers or other forms of corruption, minimax estimators provide robustness by minimizing the maximum possible loss. In the context of non-parametric estimation where the underlying distribution of the data is unknown or does not belong to any parametric family, minimax estimation provides a framework for constructing estimators that perform well under various distributional assumptions \citep{vaart1998asymptotic,tsybakov2008introduction}. In the context of hypothesis testing, minimax approaches are used to design tests that control the probability of type I and type II errors under worst-case scenarios \citep{lehmann2005testing}. Instead, in Bayesian statistics, minimax estimation can be used to derive minimax estimators under certain loss functions, providing a frequentist perspective on Bayesian inference \citep{berger1985statistical}.

The minimax of a finite set of quadratic risks without any model assumptions on the regressors or response variables (other than finite second moments) combines both aspects of non-parametric estimation and linear minimax regression. It is linear in the sense that we seek the best linear fit in a minimax least-squares sense. It is non-parametric in the sense that we make no model assumptions what-so-ever. 

In standard linear regression one assumes a linear relation between the regressors and the response variable. This assumption greatly facilitates the theory since it leads to an unbiased estimator. Nevertheless, as is often pointed out, this assumption is more flexible than it leads on at first sight, since some response-regressor relationships can be transformed into a linear one \citep{wood2017generalized,atkinson2021box}. Nonetheless, this is not applicable to all type of relationships, and, moreover, it does in such cases assume a \textit{known} relationship. 

In section~\ref{sec:theproblem} we describe the general minimax problem we are considering, namely the minimax of an affine combination of risk quadratic functions. In section~\ref{sec:estimation} we derive a consistent (a.s.) estimator of the entire solution set. Other than the theoretical value of studying this problem, it is also valuable in for instance worst-risk minimization in the context of structural equation models. In section~\ref{semsec} we provide an example of such an application. We conclude the paper in section~\ref{sec:proof} with the proof of the main theorem in the paper. 

\section{The problem}
\label{sec:theproblem}

We consider $k$ arbitrary square integrable random variables $Y_1,...,Y_k$ and $k$ arbitrary random (column) vectors of length $p$, $X_1,...,X_k$ such that $X_u(l)$ is square integrable for $1\le u\le k$ and $1\le l\le p$. We shall refer to each pairing of $X_u$ and $Y_u$ as an environment. For any row vector $\beta\in\R^p$ we form the square risk functions $R_u(\beta)=\E\left[(Y_u-\beta X_u)^2\right]$ for every environment. The main objective of this paper is to study minimax of affine combinations of quadratic risk functions of the form
\begin{align}\label{argm}
	\mathcal{B}= \arg\min_{\beta\in\R^p} \bigvee_{i=1}^m\left(\kappa_i+ \sum_{u=1}^kw_i(u) R_u(\beta)\right),
\end{align}
where $m\ge 1$, (the intercepts) $\kappa_i\ge 0$, the weights $\{w_i(u)\}_{u=1}^k$ are either non-negative, with at least one positive weight or non-positive, with at least one negative weight, for each $i=1,...,k$ and for at least one $i$ the weights are non-negative (so that whenever $\mathcal{B}$ is non-empty, any minimizer gives a positive minimum). We assume that the weights and intercepts are all known and fixed (i.e. we do not seek to estimate them). Note that $\bigvee_{i=1}^m\left(\kappa_i+\sum_{u=1}^kw_i(u) R_u(\beta)\right)$ is not always convex (unless all weights are non-negative), never concave and in general we may therefore have that the number of elements in $\mathcal{B}$ is greater than one. A special case of \eqref{argm} that is of particular interest, is when $m=k$ and $\kappa_i=0$, $w_i(i)=1$, $w_i(u)=0, i\not=u$, in this case
\begin{align}\label{classic}
	\mathcal{B}= \arg\min_{\beta\in\R^p} \bigvee_{i=1}^k R_i(\beta),
\end{align}
the classical quadratic minimax. For each environments $1\le u\le k$ we can associate a point $\theta_u\in\Theta=\R^{\left(\frac{p}{2}+1\right)(p+1)}$, where the entries in $\theta_u$ are given by the unique elements in the covariance matrix corresponding to the vector $\left(X_u(1),...,X_u(p),Y_u\right)$. Denote 
\begin{itemize}
	\item $f_i(\beta)=\kappa_i+\sum_{u=1}^kw_i(u) R_u(\beta)$, $i=1,...,m$.
	\item $f(\beta)=\bigvee_{i=1}^mf_i(\beta)$.
\end{itemize}
Since the weights and intercepts are fixed, all $f_i$'s are determined by a parameter vector $\theta^{(k)}\in \Theta^k$.

\section{Estimation of the population minimizer}
\label{sec:estimation}
Let us now set up the framework for estimation of \eqref{argm}. We assume we have $k$ i.i.d. sequences $\{(Y_{u,v},X_{u,v})\}_{v=1}^\infty$ where $Y_{u,v}\stackrel{d}{=}Y_u$ and $X_{u,v}(l)\stackrel{d}{=}X_u(l)$ for $1\le l\le p$.
Suppose we gather the first $n_u$ samples from $\{(Y_{u,v},X_{u,v})\}_{v=1}^\infty$ for each $1\le u\le k$. We denote this multi-index $\textbf{n}=\left\{n_{1},...,n_{k}\right\}$, let $\mathbb{X}_u(\textbf{n})$ be the $n_{u}\times p$ matrix whose rows are $X_{u,1},...,X_{u,n_u}$ (from top to bottom) and similarly let $\mathbb{Y}_u(\textbf{n})$ be the $n_{u}\times 1$ column vectors whose entries are $Y_{u,1},...,Y_{u,n_u}$ (from top to bottom). Let 
\begin{itemize}
	\item $G^u=\E\left[X_u^TX_u\right]$ 
	\item $\mathbb{G}^u=\frac{1}{n_{u}}\mathbb{X}_u(\textbf{n})^T\mathbb{X}_u(\textbf{n})$. 
	\item $Z^u=\E\left[X_uY_u\right]$ and
	\item $\mathbb{Z}^u=\frac{1}{n_{u}}\sum_{l=1}^{n_{u}}X_{u,l}Y_{u,l}$. 
\end{itemize}
We also let 
$$\hat{R}_u(\beta,\textbf{n})=\frac{1}{n_u}\sum_{l=1}\left(Y_{u,l}-\beta X_{u,l}\right)^2,$$
$$\hat{f}_i(\beta,\textbf{n})=\kappa_i+\sum_{u=1}^kw_i(u)\hat{R}_u(\beta,\textbf{n}),$$
and consider the set of corresponding plug-in estimators to \eqref{argm},
\begin{align}\label{plugin}
	&\hat{B}(\textbf{n})=\arg\min_{\beta\in\R^p} \bigvee_{i=1}^k \hat{f}_i(\beta,\textbf{n}).
\end{align}
Note that the search space for solutions above is all of $\R^p$, so only proving consistency does not give us a "recipe" for \textit{how} to actually find these solutions. We will start by describing a constructive set of solutions for (\ref{plugin}), for which we prove the consistency. The minimization involves solving a high-degree polynomial, which is computationally expensive. We therefore provide a computationally much more efficient estimator, which is also consistent, in section~\ref{sec:constructive}.

\subsection{Consistency}

We now provide the necessary ingredients for doing a constructive estimation of the solution set \eqref{argm}. First let us describe the types of convergences we consider for the solution sets. Let $d(x,E)=\inf\{\n y-x\n: y\in E\}$ for $x\in\R^p$, $E\subset\R^p$, where $\n.\n$ denotes the Euclidean norm. 
\begin{defin}
	We shall say that a sequence of sets $S_n\in\R^p$ converges asymptotically one-to-one to a finite set $S\in\R^p$, denoted $S_n\uparrow S$, if $d(s_n,S)\to 0$ for all $s_n\in S_n$ and there exists a $\delta>0$ such that for large enough $n$, $\n s_1-s_2 \n>\delta$ for any two distinct $s_1,s_2\in S_n$.
\end{defin}
If we have a multi-indexed sequence of sets $S_\textbf{n}$ ($\textbf{n}=(n_1,...,n_k)$) then we will use the convention that $S_\textbf{n}\uparrow S$ if $S_{n_1\wedge...\wedge n_k}\uparrow S$.
The motivation for using the term "one-to-one" comes from the fact that elements of the converging sequence of sets do not converge to the same element in $S$, so they must therefore converge to a subset of $S$. To be more formal, if $S_n\uparrow S$ there exists no subsequence $\{n_k\}_k$ and no element $s\in S$ such that $s^1_{n_k},s^2_{n_k}\in S_{n_k}$, $s^1_{n_k}\not=s^2_{n_k}$, $\n s^1_{n_k}-s\n\to 0$ and $\n s^2_{n_k}-s\n\to 0$ as $k\to\infty$. Moreover,
\begin{prop}
	If $S_n\uparrow S$ then there exists some $N$ such that for $n\ge N$, $|S_n|\le |S|$ ($|.|$ denotes the number of elements).
\end{prop}
\begin{proof}
	Suppose the opposite. Choose $n$ so large that $d(S_n,S)<\delta/2$. If $|S_n|>|S|$ there must exist (by the pigeon hole principle) $s_1,s_2\in S_n$ and $s\in S$ such that $\n s_1-s\n<\delta/2$ and $\n s_2-s \n<\delta/2$ but then $\n s_1 -s_2\n <\delta$ 
\end{proof}
To capture what we on the other hand mean by "asymptotically bijective" convergence of a sequence of sets $\{S_n\}_n$ towards a finite set $S$ (i.e. every point in $S$ is approached by exactly "one point" from the sequence $\{S_n\}_n$) we have the following definition. 
\begin{defin} For two finite subsets of $\R^p$, $A$ and $B$ we set
	$$\dist(A,B)=1_{|A|\not=|B|}+\sum_{a\in A}\inf\{\n a-b\n:b\in B\}+\sum_{b\in B}\inf\{\n a-b\n:b\in A\}. $$
	In addition if $\{A_n\}_n$ is a sequence of finite subsets in $\R^p$ then we say $A_n\longrightarrow B$ if $\dist(A_n,B)\to 0$ as $n\to\infty$.
\end{defin}
Again, if we have a multi-indexed sequence of sets $S_\textbf{n}$ ($\textbf{n}=(n_1,...,n_k)$) then we will use the convention that $S_\textbf{n}\longrightarrow S$ if $S_{n_1\wedge...\wedge n_k}\longrightarrow S$.
\begin{prop}
	If $S\subset\R^p$, $|S|<\infty$, $S_n\longrightarrow S$ then there exists some $N\in \N$ such that for $n\ge N$ we can decompose $S_n=\{s^1_n\}\cup...\cup\{s^{|S|}_n\}$ so that $\lim_{n\to\infty}s^i_n=s^i\in S$, where $s^i\not=s^j$ if $i\not=j$. While, for each $s\in S$, $\exists \hspace{1mm}\{s_n\}_n, s_n\in S_n$ such that $\lim_{n\to\infty}s_n=s$.
\end{prop}
\begin{proof}
	For large enough $n$, $\dist(S_n,S)<\infty$, so $|S|=|S_n|$ for such $n$. Since $\sum_{s_n\in S_n}\inf\{\n s_n-s\n:s\in S\}+\sum_{s\in S}\inf\{\n s-s_n\n:s_n\in S_n\}$ converges to zero we have that all points in $S_n$ converges to points in $S$, otherwise the first term would not converge to zero. Take two distinct sequences $\{s_n^1\}_n$, $\{s_n^2\}_n$ such that $s_n^1,s_n^2\in S_n$ and $s_n^1\to s_1\in S$, while $s_n^2\to s_2\in S$. Suppose that $s_1=s_2$, then the term $\sum_{s\in S}\inf\{\n s-s_n\n:s_n\in S_n\}$ will not converge to zero by the pigeon hole principle and the fact that $|S_n|=|S|$ for large $n$. This also shows the final statement of the proposition.
\end{proof} 
The following (a.s. finite for sufficiently large $\textbf{n}$) sets will be used for estimating the solution sets. Denote, 
\begin{itemize}
	\item $\hat{g}_{i,j}=\hat{f}_i-\hat{f}_j$,
	\item $\hat{B}_{inf}(\textbf{n})=\bigcup_{i=1}^m\left\{\left(\sum_{u=1}^mw_i(u)\mathbb{G}^u\right)^{-1}\sum_{u=1}^mw_i(u)\mathbb{Z}^u\right\}$,
	\item $\hat{B}_{int}(\textbf{n})=\bigcup_{1\le i<j\le m}\hat{B}^{i,j}(\textbf{n})$,
\end{itemize}
where
\begin{align}\label{Bij}
	\hat{B}^{i,j}(\textbf{n})=\left\{\hat{M}^{-1}(\lambda^{i,j})\hat{C}^{i,j}_{\lambda^{i,j}}: \lambda^{i,j}\in\arg\min_{\lambda:\hat{g}_{i,j}\left( M^{-1}(\lambda)C_{\lambda}\right)=0 }\hat{f}_i\left((\hat{M}^{i,j})^{-1}(\lambda)\hat{C}^{i,j}_\lambda\right)\right\},
\end{align}
if $\hat{f}_i$ and $\hat{f}_j$ intersect otherwise we set $\hat{B}^{i,j}(\textbf{n})=\emptyset$. The definitions of $\hat{M}(\lambda)$, $\hat{C}(\lambda)$ are given in the appendix. We now consider the argmin over the following finite  set of candidates ,
\begin{align}\label{betaeqemp}
	&\hat{B}(\textbf{n})=\arg\min_{\beta\in \left(\hat{B}_{inf}(\textbf{n}) \cup \hat{B}_{int}(\textbf{n})\right)\cap\{\beta\in\R:\hspace{1mm} \exists 1\le i \le k, \hat{f}(\beta)= \hat{f}_i(\beta)\}} \hat{f}(\beta) .
\end{align}
We also define the (again, the number of elements in this set is finite if $\textbf{n}$ is sufficiently large) set of candidates within an $\epsilon$-distance of the plug-in argmin
$$\hat{B}(\textbf{n})_\epsilon=\{\beta \in \hat{B}_{inf}(\textbf{n})\cup \hat{B}_{int}(\textbf{n}): \hat{f}(\beta)> \hat{f}(\tilde{\beta})-\epsilon, \tilde{\beta}\in\hat{B}(\textbf{n})\}.$$
We will now state the consistency of the estimator. However, the theorem proves more than just consistency, as it also circumvents the issue of the infinite search space $\mathbb{R}^p$ by using a finite set of candidates in $\hat{B}(\textbf{n})$ and $\hat{B}(\textbf{n})_\epsilon$. 

\begin{thm}\label{optimalregthm}
	Outside a subset of Lebesgue measure zero $N\in\Theta^k$ of choices of $\theta^{(k)}$, the following statements hold true. $\mathcal{B}$ is a non-empty finite set in $\R^p$.  $\hat{B}(\textbf{n})\uparrow\mathcal{B}$ and there exists an $\epsilon>0$ (possibly depending on the choice of the $\theta^{(k)}$s) such that $\hat{B}(\textbf{n})_{\epsilon'}\longrightarrow \mathcal{B}$ a.s., for any $\epsilon'\le \epsilon$. In the special case \eqref{classic}, $|\mathcal{B}|=1$ (i.e. there is a unique minimizer) and $\hat{B}(\textbf{n})\longrightarrow \mathcal{B}$ (and this holds outside a set of measure zero in $\Theta\times\Theta$).  
\end{thm}
\begin{rema}
	It is possible in the special case \eqref{classic} that if $\textbf{n}$ is not sufficiently large then $\hat{f}$ is not strictly convex and may therefore have several minimizers.
\end{rema}
\subsection{The approximate plug-in estimator}
\label{sec:constructive}
From a practitioners point of view there is an issue with the estimators in the above theorems. Namely it requires the computation of the roots of a polynomial of degree $p$. By the Abel-Ruffini theorem we can't solve general polynomials of this type in terms of radicals. There is however an approximate plug-in estimator (or in reality a family of estimators) that will also be consistent, which do not require us to compute the roots analytically. The solution lies in the proof of Theorem \ref{optimalregthm}, when computing the roots in \eqref{groots} we note that they are simple roots so we may approximate these roots by using the bisection method. We now describe how to compute this approximate plug-in estimator.
\\
\\
The following procedure applies outside a set of Lebesgue measure zero in $\Theta^k$ ($\Theta\times\Theta$ in the case of \eqref{classic}). Let $\{c_m\}_m$ be any sequence in $\N$ such that $c_m\to\infty$, assume we have $\textbf{n}=(n_1,...,n_k)$ samples from every environment and define $n=n_1\wedge...\wedge n_k$. We will refer to the definitions of $\tilde{P}(\lambda)$, $\hat{\tilde{P}}(\lambda)$ and $\hat{\beta}(\lambda)$ from the proof of \ref{optimalregthm}. If $n$ is sufficiently large then all the roots (in terms of $\lambda$) to $\hat{\tilde{P}}(\lambda)=0$ are simple (if the distributions of $X_u(l)$ and $Y_u$, $u=1,...k$, $l=1,...,p$ are all absolutely continuous then it is true for all $n$), which we will assume going forward. Let $\bar{\beta}(\lambda)$ be the approximation of $\hat{\beta}(\lambda)$ where we replace the roots of $\hat{\tilde{P}}$ with the following approximation. Since $\tilde{P}(\lambda)$ is a polynomial of degree $p-1$ we may write $\tilde{P}(\lambda)=\sum_{u=0}^{p-1}e_u\lambda^u$ it then follows that $\hat{\tilde{P}}(\lambda)=\sum_{u=0}^{p-1}\hat{e}_u\lambda^u$, where $\hat{e}_u$ is the plug-in estimator of $e_u$. By the Lagrange bound all real roots of $\hat{\tilde{P}}$ can be contained in $\left[-\left(1\vee\sum_{u=1}^{p-2}\left|\frac{\hat{e}_{u}}{\hat{e}_{p-1}}\right|\right),1\vee\sum_{u=1}\left|\frac{\hat{e}_{u}}{\hat{e}_{p-1}}\right|\right]$. Let $R_n=1\vee\sum_{u=1}^{p-2}\left|\frac{\hat{e}_{u}}{\hat{e}_{p-1}}\right|$, then since $\hat{e}_{u}\to e_u$ by the law of large numbers for $u=0,...,p-1$ we have that $R_n\xrightarrow{a.s.}R$ where $R=1\vee\sum_{u=1}^{p-2}\left|\frac{e_{u}}{e_{p-1}}\right|$. So for large $n$, $[-R_n,R_n]\subset[-(R+1),R+1]$. By Theorem 1 in \cite{rump1979polynomial} we have that the minimal distance between any roots of $\tilde{P}$ is bounded below by 
$$\Delta:= \left(1\vee\left|e_{p-1}\right|\right)^{(p-1)(\ln (p-1)+1)}D(\tilde{P})\frac{(2(p-1))^{p-2}}{s^{(p-1)(\ln(p-1)+3)}},$$
where $D(\tilde{P})$ denotes the discriminant of $P$ and $s=\sum_{u=0}^{p-1}|e_u|$. Similarly we let 
$$\hat{\Delta}= \left(1\vee\left|\hat{e}_{p-1}\right|\right)^{(p-1)(\ln (p-1)+1)}D(\hat{\tilde{P}})\frac{(2(p-1))^{p-2}}{\hat{s}^{(p-1)(\ln(p-1)+3)}},$$
then clearly $\hat{\Delta}\xrightarrow{a.s.}\Delta$ by continuity and the law of large numbers. For sufficiently many samples, the discriminant of $\hat{\tilde{P}}(\lambda)$ must be non-zero (by continuity). Let $\epsilon>0$, we now divide $[-R_n,R_n]$ into $m_n:=2\left \lceil{\frac{R_n}{\hat{\Delta}+\epsilon}}\right \rceil $ intervals $\{I_u\}_{u=1}^{m_n}$ of equal length such that $|I_u|<\Delta$. Any such interval $I_u$ has a root to $\hat{\tilde{P}}$ if and only if the two endpoints of $I_u$ are of opposite sign. This is because all the roots are simple and will therefore incur a sign change in $f$. Consider the $v$:th interval with a sign change for $\hat{\tilde{P}}$, this interval must contain $\hat{\lambda_v}$ (the $v$:th smallest real root to $\hat{\tilde{P}}$), we compute $c_n$ number of bisections to get our approximation $\bar{\hat{\lambda}}_v$ of $\hat{\lambda}$, which then will have the property $\left| \bar{\hat{\lambda}}_u - \hat{\lambda}_u\right|<\frac{1}{2^{c_n}}$ and therefore $\bar{\hat{\lambda}}_v\xrightarrow{a.s.}\lambda_v$. 

\begin{rema}
	Since the roots are simple above, one may of course apply faster algorithms (such as Budans method) than the plain bisection method. 
\end{rema}

\subsubsection{Algorithm for computing the approximate plug-in estimator}
We now describe exactly how to compute the approximate plug-in estimator given $n_i$ samples from environment $1\le i\le k$.
\begin{itemize}
	\item
	Compute the inflexion point of to every $\hat{f}_i(\beta)$ which is given by 
	$$\left(\sum_{u=1}^mw_i(u)\mathbb{G}^u\right)^{-1}\sum_{u=1}^mw_i(u)\mathbb{Z}^u.$$
	\item
	For every $1\le i<j\le m$ such that $\hat{f}_i(\beta)$ and $\hat{f}_j(\beta)$ intersect, compute all the roots to 
	$$\hat{\tilde{P}}(\lambda)=\det\left(\hat{M}^{i,j}(\lambda)\right)^2\hat{g}_{i,j}(\hat{\beta}(\lambda)),$$ 
	(where $\hat{g}_{i,j}(\beta(\lambda))$ is given by \eqref{groots}) using the bisection method outlined above. For every such root $\lambda$, compute $\hat{\beta}(\lambda)=\left(\hat{M}^{i,j}\right)^{-1}(\lambda)\hat{C}^{i,j}(\lambda)$
	\item
	The $\arg\min$ solutions are now given by those $\beta$ amongst the candidates computed above that minimize
	$$\hat{f_1}(\beta)\vee...\vee\hat{f_m}(\beta).$$
\end{itemize}
The following theorem summarizes our findings in this section.
\begin{thm}
	Let $\hat{\bar{B}}(\textbf{n})$ denote the set of estimators described above and let $\hat{\bar{B}}(\textbf{n})_\epsilon$ be defined analogously to $\hat{B}(\textbf{n})_\epsilon$ in the beginning of section 3.1. Then outside at set of measure zero of choices of $\theta^{(k)}\in\Theta$. $\hat{\bar{B}}(\textbf{n})\uparrow \mathcal{B}$ and $\hat{B}(\textbf{n})_\epsilon\longrightarrow \mathcal{B}$ for sufficiently small $\epsilon$.
\end{thm}

\section{An application to structural equation models (SEMs)}\label{semsec}
Assume we are given a probability space $\left(\Omega,\F,\P \right)$. We consider structural equation models (SEMs) of the following type 
\begin{equation}\label{SEMA}
	\begin{bmatrix}
		Y^{A}\\
		X^{A}
	\end{bmatrix} = 
	B(\omega)\cdot\begin{bmatrix}
		Y^{A}\\
		X^{A}
	\end{bmatrix}
	+\epsilon_{A}
	+A
	,
\end{equation}
where $B(\omega)$ is a random real-valued $(p+1)\times(p+1)$ matrix. The random vector $A\in\R^p$ is called a shift and $Y^{A}$ and $X^{A}$ are the target and covariates associated to the environment shifted by $A$. The vectors $\epsilon_{A}\in\R^p$ are assumed to have the same distribution for all $A$ and are assumed to be conditionally uncorrelated with the shifts given $B$, i.e. $\E\left[\epsilon_{A}A^T \left|\right.B\right]=0$. We assume that $I-B$ has full rank a.s., which implies that $X^{A}$ and $Y^{A}$ have unique solutions for every $A$. Note that due to the fact that $B$ is random, this SEM is not linear. Let us now assume we are given $k\in\N$ shifts, $A_1,...,A_k$ and that $Y^{A_i}$ and $X^{A_i}$ are square integrable for each $i$. In addition we assume that $Y^{0}$ and $X^{0}$ (i.e. zero shifted) are also square integrable. Now we consider the space of shifts 
$$C\left(A_1,...,A_k\right)=\left\{A\in\mathcal{A}:\exists 1\le i\le k, \E\left[AA^T|B\right]\preccurlyeq \E\left[A_iA_i^T | B\right]\mathit{ a.s.}\right\},$$
where $\mathcal{A}$ is any set of shifts that contains $\bigcup_{i=1}^k\{A_i\}$. We define the risk associated with shift $A$ by $R_A(\beta)=\E\left[\left(Y^A-\beta X^A\right)^2\right]$. With this notation we have the following result that characterizes the worst risk over the entire shift-space $C\left(A_1,...,A_k\right)$.
\begin{prop}\label{suppropfixedw}
	$$\sup_{A\in C\left(A_1,...,A_k\right)}R_{A}(\beta)=\bigvee_{i=1}^kR_{A_i}(\beta).$$
\end{prop}
\begin{proof}
	Since $C\left(A_1,...,A_k\right)=\bigcup_{i=1}^k C\left(A_i\right)$ it follows that
	$$\sup_{A\in C\left(A_1,...,A_k\right)}R_{A}(\beta)=\bigvee_{i=1}^k\sup_{A\in C\left(A_i\right)}R_{A}(\beta).$$
	Therefore it suffices to show that $\sup_{A\in C\left(A_i\right)}R_{A}(\beta)=R_{A_i}(\beta)$. Recall that all the noise has the same distribution across all environments, we shall denote by $\epsilon$, a generic random vector with such a distribution. Since the entries of $(I-B)^{-1}$ are $\sigma(B)$-measurable it follows that if we define $v=\left((I-B)^{-1}_{1:p,\textbf{.}}\beta-((I-B)^{-1})_{p+1,\textbf{.}}\right)$ then $v$ is also $\sigma(B)$-measurable. With this notation
	$$\sup_{A\in C\left(A_i\right)}R_{A}(\beta) =\sup_{A\in C\left(A_i\right)}\E\left[ v(A+\epsilon_{A})(A+\epsilon_{A})^T v^T\right].$$
	Since $Y^{0}$ and $X^{0}$ are also square integrable, we have that $R_0(\beta)<\infty$ for any $\beta\in\R^p$. It then follows that $\E\left[v\epsilon\epsilon^Tv^T\right]<\infty$. Similarly, since $R_{A_i}(\beta)<\infty$ it follows that $\E\left[ v(A_i+\epsilon)(A_i+\epsilon)^T v^T\right]<\infty$. We also have $\E\left[ vA_i\epsilon^T v^T\right]=\E\left[ v\E\left[A_i\epsilon^T|B\right] v^T\right]=0$ and hence $\E\left[vA_iA_i^Tv^T\right]<\infty$ for all $1\le i\le k$. This leads to
	\begin{align}\label{supeq1}
		\sup_{A\in c(A_i)}R_{A}(\beta) &=\E\left[ v\epsilon\epsilon^T v^T\right]+\sup_{A\in c(A_i)}\left(2\E\left[ v\epsilon A^T v^T\right]+\E\left[ vAA^T v^T\right]\right)
		\nonumber
		\\
		&=\E\left[ v\epsilon\epsilon^T v^T\right]+\sup_{A\in c(A_i)}\E\left[ vAA^T v^T\right].
	\end{align}
	If $A\in C(A_i)$ then
	\begin{align*}
		\E\left[vAA^Tv^T\right]
		&=
		\E\left[\E\left[vAA^Tv^T|B\right]\right]
		=
		\E\left[v\E\left[AA^T|B\right]v^T\right]
		\le
		\E\left[v\E\left[A_iA_i^T|B\right]v^T\right]
		\\
		&=\E\left[\E\left[vA_iA_i^Tv^T|B\right]\right]
		=\E\left[vA_iA_i^Tv^T\right],
	\end{align*}
	i.e. $\sup_{A\in c(A_i)} \E\left[vAA^Tv^T\right]\le \E\left[vA_iA_i^Tv^T\right]$. Since $A_i\in C(A_i)$ it follows that 
	$$\sup_{A\in C(A_i)} \E\left[vAA^Tv^T\right]=\E\left[vA_iA_i^Tv^T\right].$$
	Going back to \eqref{supeq1} we find
	\begin{align}\label{supeq}
		\sup_{A\in C(A_i)}R_{A}(\beta)
		&=
		\E\left[v\epsilon\epsilon^Tv^T\right]+\E\left[vA_i A_i^Tv^T\right]
		=R_{A_i}(\beta),
	\end{align}
	which is what we wanted to prove.
\end{proof}

\section{Proof of Theorem \ref{optimalregthm}}
\label{sec:proof}

Before the proof of Theorem \ref{optimalregthm} we will need to construct a sizeable toolbox in the form of several Lemmas that will now follow.
For each $\theta^{(k)}\in\Theta^k$ we may bijectively associate a pairing of an affine (and symmetric) matrix function and an affine row vector. These will be tools that we use to tackle the argmins along the intersection between two affine combinations, $f_i,f_j$ with given parameter set $\theta^{(k)}\in\Theta^k$. First, given $\theta^{(k)}\in\Theta^k$ denote (recall that by assumption for each $1\le i\le m$ there exists $1\le u\le k$ such that $w_i(u)\not=0$)
\begin{itemize}
	\item $a_i(l_1,l_2)=\sum_{u=1}^kw_i(u)\E\left[X_u(l_1)X_u(l_2)\right]$,
	\item $a_i(l)=\sum_{u=1}^kw_i(u)\E\left[X_u(l)^2\right]$,
	\item $b_i(v)=\sum_{u=1}^kw_i(u)\E\left[X_u(v)Y_u\right]$ and
	\item $c_i=\kappa_i+\sum_{u=1}^kw_i(u)\E\left[Y_u^2\right]$.
\end{itemize}
With this notation we are ready for our main definition.
\begin{defin}\label{parammatrix}
	For every $\theta^{(k)}\in\Theta^k$ let the $p\times p$ \textit{affine covariate matrix},  $M^{i,j}(\lambda)$ and the \textit{affine target vector} $C^{i,j}(\lambda)$ be defined element wise by, when $i$ is such that $f_i$ has non-negative weights
	$$C^{i,j}(\lambda)_u=b_i(u)-\lambda\left(b_i(u)-b_j(u)\right),$$
	$$M^{i,j}(\lambda)_{l,l} =a_i(l)-\lambda(a_i(l)-a_j(l))).$$
	When $u\not=v$
	$$M^{i,j}(\lambda)_{l_1,l_2} =a_i(l_1,l_2)-\lambda \left( a_i(l_1,l_2)-a_j(l_1,l_2)\right).$$
	If $i$ is such that $f_i$ has non-positive weights we simply change the sign in front of $\lambda$.
\end{defin}
We may regard $\det\left(M^{i,j}(\lambda)\right)$ as polynomial in $\R$ with coefficients in $\Theta^k$ (as we view the weights and intercepts as fixed). Doing so we will denote the roots in terms of $\lambda$ as $\Lambda(\theta^{(k)})$ where we highlight the dependence on $\theta^{(k)}\in\Theta^k$. We will denote the real roots of $\det\left(M^{i,j}(\lambda)\right)$ as $\Lambda(\theta^{(k)})_\R$. Often times we will suppress the dependence on $\theta^{(k)}$ for brevity when we see fit.
The following result comes from complex analysis.
\begin{lemma}\label{realsimple}
	Let $P(\lambda)=\sum_{u=0}^ma_u\lambda^u$ be a polynomial with real coefficients ($a_m\not=0$) with a real simple root $r$. Let $r_1,...,r_v$ denote all the distinct roots of $P$. For any $0<\epsilon<\frac 12\min_{1\le u<z\le v}|r_u-r_z|$ there exists $\delta>0$ such that if we consider any polynomial of the form $\tilde{P}(\lambda)=\sum_{u=0}^mb_u\lambda^u$ with $|a_u-b_u|<\delta$ then $\tilde{P}$ must have a simple real root in $(r_u-\epsilon,r_u+\epsilon)$.
\end{lemma}
The above result is then applied to get the following lemma which is what we actually will need.
\begin{lemma}\label{simpleroots}
	Let $P(\lambda)=\sum_{u=0}^ma_u\lambda^u$ be a polynomial with real coefficients ($a_m\not=0$), with $v$ distinct roots and whose real roots $r_1<...<r_j$ ($j\le v$) are simple. If $P_n(\lambda)=\sum_{u=0}^ma_{u,n}\lambda^u$ is a sequence of polynomials such that $a_{u,n}\to a_u$, $u=0,...,m$ then if we denote the real roots of $P_n(\lambda)$ as $r_1(n)<...<r_{w_n}(n)$ then $w_n=j$ for large enough $n$ and $r_u(n)\to r_u$ for $u=1,...,j$.
\end{lemma}
\begin{proof}
	We know by Lemma \ref{realsimple} that for any $0<\epsilon<\frac 12\min_{1\le u<z\le v}|r_u-r_z|$ there exists $\delta>0$ such that if $|a_u-a_{u,n}|<\delta$ then $P_n$ must have a simple real root in $(r_u-\epsilon,r_u+\epsilon)$, which shows that $w(n)\ge v$ for large enough $n$ and that $r_u(n)\to r_u$ for $u=1,...,j$. If $j=v$ we are done, since then all roots are real and simple. Suppose instead that $j<v$. We know that all roots of $P_n$ must converge to those of $P$ so take any root $r$ of $P$ with non-zero imaginary part $c=\mathrm{Im}(r)$ then we know that for large enough $n$, $P_n$ must have a root $r'(n)$ such that $|r-r'(n)|<c/2$ which implies that the imaginary part of $r'(n)$ is non-zero for such $n$. But this is true for all roots of $P$ that are not real and since $P_n$ has the same degree as $P$ for large enough $n$ (when $a_{m,n}\not=0$), $P_n$ must have the same number of roots that are not real as $P$ (and therefore the same number of roots that \textit{are} real), i.e. $w_{(n)}=j$ for such $n$.
\end{proof}
A result from measure theory which will lie at the heart of the method by which we prove Theorem \ref{optimalregthm} is the following.
\begin{lemma}\label{PolyLemma}
	A polynomial on $\R^n$, for any $n\in\N$ is either identically zero or is only zero on set of Lebesgue measure zero.
\end{lemma}
The following lemma is an immediate application of the one above and illustrates how we will apply the above lemma in this paper.
\begin{lemma}\label{discriminantLemma}
	Any non-zero polynomial $P_{\theta^{(k)}}\left(\lambda\right)$ on $\R\times\Theta^k$ has simple roots (in terms of $\lambda$) outside a set of measure zero in $\Theta^k$.
\end{lemma}
\begin{proof}
	$P$ only has simple roots if the discriminant is non-zero. The discriminant is a (non-zero) polynomial on $\Theta^k$ and therefore is only zero on a set of measure zero in $\Theta^k$.
\end{proof}
We now prove some important properties for $M^{i,j}$.
\begin{lemma}\label{matrixLemma}
	The following results holds for the affine parameter matrix $M^{i,j}(\lambda)$ as defined in \ref{parammatrix}.
	\begin{itemize}
		\item [1] Outside a set $N$ of measure zero in $\Theta^k$, there are at most $p$ elements in $\Lambda$ and for any $\lambda\in\Lambda$,
		$rank\left(M^{i,j}(\lambda)\right)=p-1$.
		\item[2] Moreover for any such $\lambda$ there is a unique set of $p-1$ real numbers $s_1(\theta^{(k)},\lambda),...,s_{p-1}(\theta^{(k)},\lambda)$, all of which must be non-zero and such that 
		\begin{align}\label{coeffrepreq}
			&M^{i,j}_{p,.}(\lambda)=\sum_{u=1}^{p-1}s_u(\theta^{(k)},\lambda)M^{i,j}_{u,.}(\lambda),
		\end{align}
		outside $N$.
	\end{itemize}
	
\end{lemma}
\begin{proof}
	Denote $P(\lambda)=\det\left(M^{i,j}(\lambda)\right)$, by the fundamental theorem of algebra, $P$ has at most $p$ distinct solutions in $\mathbb{C}$. The rank of $M^{i,j}(\lambda)$ is $p-1$ if and only if there exists some non-zero minor, that is to say there exists a minor that has no roots in common with $P$. Consider a submatrix $M_{u,v}(\lambda)$ (where the subscripts $u,v$ indicate that we remove row $u$ and column $v$) of $M^{i,j}(\lambda)$, and let $P_{u,v}(\lambda)=\det\left(M_{u,v}(\lambda)\right)$. We wish to show that outside a set of measure zero in $\Theta^k$, there exists a minor such that $P$ and $P_{u,v}$ has no roots in common. $P$ and $P_{u,v}$ has a root in common in $\mathbb{C}$ if and only if their resultant is zero (since $\mathbb{C}$ is an algebraically closed field). But the resultant is a polynomial on $\Theta^k$ so by Lemma \ref{PolyLemma} this polynomial is either the zero polynomial or it is zero on a set of measure zero in $\Theta^k$. So to prove the first statement of the theorem, it suffices to show that there exists any $\theta^{(k)}\in\Theta^k$ such that the resultant of $P$ and $P_{u,v}$ is not zero. Recall however that this is true if and only if for some $\theta^{(k)}$, $P(\lambda)=0$ implies $P_{u,v}(\lambda)\not=0$ for some $u,v\in\{1,...,p\}$ and all $\lambda\in\R$. For this purpose consider the matrix $M'(\lambda)$ where 
	%
	\begin{equation*}
		M'(\lambda) = 
		\begin{pmatrix}
			\lambda & \lambda-1 &0 & \cdots & 0 \\
			0 & 1 & 0 & \cdots & 0  \\
			\vdots  & \vdots  & \ddots & \vdots & \vdots \\
			0 & \cdots & 0 & 1 & 0  \\
			0 & \cdots & 0 & 0 & 1
		\end{pmatrix}
	\end{equation*}
	In this case the only root to $P$ is given by $\lambda=0$ while the only root for $P_{1,2}$ is given by $\lambda=1$. This proves (1). As for (2), it is enough to show that every first minor of the form $\det\left(M^{i,j}_{u,1}(\lambda)\right)$ for $u=1,...,p-1$ is non-zero since then $M^{i,j}_{p,.}(\lambda)$ must be linearly independent from all subsets of $n-2$ other rows. $\det\left(M^{i,j}_{1,u}(\lambda)\right)=0$ for a root of $P$ if and only if $res(P_{u,v},P)$ is zero. Again it suffices to show that for each $1\le u\le p-1$ there exists any $\theta^{(k)}\in\Theta^k$ such that $P$ and $P_{u,1}$ have no common root. Letting $\theta$ be such that $M^{i,j}(\lambda)$ is the diagonal matrix with $\lambda$ as diagonal element number $u$ and all other diagonal entries one implies that $P$ only has the root $0$ and $P_{u,1}$ has no roots at all. This completes the proof of the final claim.
\end{proof}
\begin{lemma}\label{matrixNoTiesLemma}
	The coefficients $\{s_u(\theta^{(k)},\lambda)\}_{u=1}^{p-1}$ in Lemma \ref{matrixLemma} are rational functions on $\R\times\Theta^k$, whose roots in terms of $\lambda$ do not lie in $\Lambda$ outside some zero set in $\Theta^k$.
\end{lemma}
\begin{proof}
	~~\\For any root of $\lambda\in\Lambda$, let $T_1(M^{i,j}(\lambda))$ denote the matrix resulting from subtracting  $\frac{M^{i,j}(\lambda)(u,1)}{M^{i,j}(\lambda)(1,1)}M^{i,j}_{1,.}$ from row number $u$, $u=2,...,p$ in $M^{i,j}(\lambda)$, for those $\theta^{(k)}\in\Theta^k$ such that $M^{i,j}(\lambda)(1,1)\not=0$, $\forall \lambda\in\Lambda(\theta^{(k)})$. For $1<v<p$, let $T_v(M^{i,j}(\lambda))$ denote the matrix resulting from subtracting\\ $\frac{T_{v-1}(M^{i,j}(\lambda))(v,u)}{T_{v-1}(M^{i,j}(\lambda))(v,v)}T_{v-1}(M^{i,j}(\lambda)_{v,.})$ (whenever $T_{v-1}(M^{i,j}(\lambda))(v,v)\not=0$) from row number $u$, $u=v+1,...,p$ in $T_{v-1}(M^{i,j}(\lambda))$. Then outside a zero set in $\Theta^k$, $T_vM^{i,j}(\lambda)$ is well-defined for $v=1,...,p-1$ and $(T_vM^{i,j}(\lambda))_{p,.}=(0,...,0)$.
	The following procedure shows us that outside a zero set in $\Theta^k$ we may bring $M^{i,j}(\lambda)$ to a row echelon form in $p-1$ steps, without permuting any rows or columns while making the final row the only zero-row in this matrix. First we see that $d_1(\lambda)=(M^{i,j}(\lambda))(1,1)$ and $P(\lambda)$ only have a common root if their resultant is zero. As before it suffices to find $\theta^{(k)}\in\Theta^k$ such that $d_1(\lambda)\not=0$ for any $\lambda\in\Lambda(\theta^{(k)})$. Consider the matrix $\tilde{M}$ with $\tilde{M}(l,l)=1$ for $l=1,...,p-1$, $\tilde{M}(p,p)=p-\lambda$, $\tilde{M}(p,l)=1$ for $l=1,...,p-1$, $\tilde{M}(1,l)=1$ for $l=1,...,p-1$ and zero on all other entries. 
	\begin{equation*}
		\tilde{M}(\lambda) = 
		\begin{pmatrix}
			1 & 0 & \cdots & 0 & 1 \\
			0 & 1 & \cdots & 0 & 1 \\
			\vdots  & \vdots  & \ddots & \vdots & \vdots \\
			0  & 0  & \cdots & 1 & 1  \\
			1 & 1 & \cdots & 1 & \lambda+p 
		\end{pmatrix}
	\end{equation*}
	Then $d_1\equiv1$ and $\Lambda(\theta^{(k)})=\{0\}$ and this shows that $d_1$ and $P$ do not have common roots outside some zero set $N_1$, so that we may divide by $d_1$. On $N_1$, $\frac{1}{d_1}=\frac{1}{M^{i,j}(1,1)}$ is well defined and in order to clear the first column below row 1 we want to ensure that we do not clear the next element below on the diagonal ($(M^{i,j}(\lambda))(2,2)$) when doing so. We therefore have to check that $\frac{(M^{i,j}(\lambda))(2,1)}{(M^{i,j}(\lambda))(1,1)}(M^{i,j}(\lambda))(1,2)\not=(M^{i,j}(\lambda))(2,2)$, which is implied if $(M^{i,j}(\lambda))(2,1)(M^{i,j}(\lambda))(1,2)-(M^{i,j}(\lambda))(2,2)(M^{i,j}(\lambda))(1,1)\not=0$. $(M^{i,j}(\lambda))(2,1)(M^{i,j}(\lambda))(1,2)-(M^{i,j}(\lambda))(2,2)(M^{i,j}(\lambda))(1,1)$ is a polynomial on $\Theta^k$ and we wish to show it has no roots in common with $P$ outside a zero set $N_2\in\Theta^k$. It suffices to find a $\theta^{(k)}\in\Theta^k\setminus N_1$ such that the two polynomials do not have a root in common. As $\tilde{M}(2,1)\tilde{M}(1,2)-\tilde{M}(2,2)\tilde{M}(1,1)=-1$, the matrix $\tilde{M}$ still suffices for this purpose. This also shows that all entries in $T_1(M^{i,j}(\lambda))$ are rational functions on $\R\times\Theta^k$. Suppose now that $1<v\le p-1$, $(T_{v-1}(M^{i,j}(\lambda))(v,v)\not=0$, all entries in $(T_{v-1}(M^{i,j}(\lambda))$ are rational on $\R\times\Theta^k$  and
	\begin{align}\label{dia}
		&\frac{T_{v-1}(M^{i,j}(\lambda))(v,v-1)}{(T_{v-1}(M^{i,j}(\lambda))(v-1,v-1)}T_{v-1}\left(M^{i,j}(\lambda)\right)(v-1,v)\not=T_{v-1}\left(M^{i,j}(\lambda)\right)(v,v),
	\end{align}
	outside a zero set $N_{v-1}\in\Theta^k$. This means that we have successfully cleared everything below the main diagonal until the $v-1$:th element on the diagonal counting from the upper left corner. The elements of $(T_{v}(M^{i,j}(\lambda)))$ are obviously rational functions on $\R\times\Theta^k$ due to the induction hypothesis and the fact that the elements in $(T_{v}(M^{i,j}(\lambda)))$ are formed by products of rational functions in $(T_{v-1}(M^{i,j}(\lambda)))$. Moreover it has no poles in $\Lambda$, outside $N_{v-1}$ (this follows from \eqref{dia}), so the roots of $(T_{v}(M^{i,j}(\lambda))(v,v)$ coincides with those of its numerator polynomial, $Q_v$. To show $Q_v$ is not the zero polynomial take the above $\tilde{M}$ again and note that $(T_{v}(\tilde{M})(v,v)=\tilde{M}(v,v)$ for $v=1,...,p-1$. To show that
	$$\frac{T_{v}(M^{i,j}(\lambda))(v+1,v)}{(T_{v}(M^{i,j}(\lambda))(v,v)}T_v(M^{i,j}(\lambda))(v,v+1)\not=T_v(M^{i,j}(\lambda))(v+1,v+1),$$
	for $\lambda\in\Lambda$, it suffices to show that the numerator polynomial of 
	$$T_{v}(M^{i,j}(\lambda))(v+1,v)T_{v}(M^{i,j}(\lambda))(v,v+1)-T_{v}(M^{i,j}(\lambda))(v+1,v+1)T_{v}(M^{i,j}(\lambda))(v,v),$$
	does not have a root in common with $P$, i.e., their resultant is only zero on a zero set. For this purpose we can again recycle $\tilde{M}$
	where we have that 
	$$T_{v}(\tilde{M})(v+1,v)T_{v}(\tilde{M})(v,v+1)-T_{v}(\tilde{M})(v+1,v+1)T_{v}(\tilde{M})(v,v)=-1,$$
	similarly to before, when $v<p$. We then let $N'_{v}$ denote the zero set where either $Q_v(\lambda)=0$ or $\frac{T_{v}(M^{i,j}(\lambda))(v+1,v)}{(T_{v}(M^{i,j}(\lambda))(v,v)}M^{i,j}_{v,v+1}=M^{i,j}(v+1,v+1)$ and then let $N_v=N_1\cup....\cup N_{v-1}\cup N_v'$ which is the desired zero set. Note that there will be two rational functions $e_{p-1},e_p$ on $\R\times\Theta^k$ such that
	\begin{equation*}
		\left(T_{p-2}(M^{i,j}(\lambda))\right)_{p,.} = 
		\begin{pmatrix}
			0 & \cdots & 0 & e_{p-1} & e_p \\
		\end{pmatrix},
	\end{equation*}
	i.e. the only two possible non-zero elements of $\left(T_{p-2}(M^{i,j}(\lambda))\right)_{p,.}$ are the last two elements. Now we must have that
	\begin{equation*}
		\left(T_{p-1}(M^{i,j}(\lambda))\right)_{p,.} = 
		\begin{pmatrix}
			0 & \cdots & 0 \\
		\end{pmatrix},
	\end{equation*}
	indeed, since if this was not the case we would have that $\mathrm{rank}\left(T_{p-1}(M^{i,j}(\lambda))\right)=p$ (since we have all elements on the main diagonal are non-zero and all elements below it have been cleared), but $\mathrm{rank}\left(T_{p-1}(M^{i,j}(\lambda))\right)=\mathrm{rank}\left(M^{i,j}(\lambda)\right)=p-1$ since $\lambda\in \Lambda$.
	The fact that $T_{p-1}(M^{i,j}(\lambda))_{p,.}=0$ leads to the following equation.
	\begin{align}\label{Mp}
		&M^{i,j}(\lambda)_{p,.}=\frac{(M^{i,j}(\lambda))(p,1)}{(M^{i,j}(\lambda))(1,1)} M^{i,j}(\lambda)_{1,.}+\frac{T_1(M^{i,j}(\lambda))(p,2)}{T_1(M^{i,j}(\lambda))(2,2)} T_1(M^{i,j}(\lambda))_{2,.}+....
		\nonumber
		\\
		&+\frac{T_{p-2}(M^{i,j}(\lambda))(p,p-1)}{T_{p-2}(M^{i,j}(\lambda))(p-1,p-1)} T_{p-2}(M^{i,j}(\lambda))_{p-1,.}
	\end{align}
	$T_1(M^{i,j}(\lambda))_{2,.}$ is a linear combination of $(M^{i,j}(\lambda))_{1,.}$ and $(M^{i,j}(\lambda))_{2,.}$ with coefficients that are rational on $\R\times\Theta^k$. Similarly $T_v(M^{i,j}(\lambda))_{v,.}$ is a linear combination of $\{(M^{i,j}(\lambda))_{u,.}\}_{u=1}^v$, for $v=2,...,p-1$, with coefficients that are rational on $\R\times\Theta^k$. This together with \eqref{Mp} and the uniqueness of the representation \eqref{coeffrepreq} shows that the coefficients $\{s_u\}_{u=1}^{p-1}$ are rational on $\R\times\Theta^k$.
\end{proof}

\begin{lemma}\label{coefficientslemma}
	Suppose $A$ is a $p\times p$ matrix of rank $p-1$ such that $A_{p,.}$ is linearly independent of every subset of $p-2$ other rows of $A$. Let $A_{p,.}=\sum_{u=1}^{p-1}c_uA_{u,.}$ be its unique decomposition in $span\{A_{1,.},...,A_{p-1,.}\}$. If $\{A(n)\}_n$ is a sequence of $p\times p$ matrices of rank $p-1$ that also share the property that $A_{p,.}(n)$ is linearly independent of every subset of $p-2$ other rows of $A(n)$ and such that $\lim_{n\to\infty}\n A-A(n)\n=0$ (for any matrix norm $\n.\n$) then if we let $A_{p,.}(n)=\sum_{u=1}^{p-1}c_u(n)A_{u,.}(n)$ be its unique decomposition in $span\{A_{1,.}(n),...,A_{p-1,.}(n)\}$   it follows that $c_u(n)\to c_u$
\end{lemma}
\begin{proof}
	Let $x=A_{p,.}$ and $V=\left(A_{1,.}^T A_{2,.}^T ... A_{p-1,.}^T\right)$ (which is then a $p\times (p-1)$ matrix) so that $x=VC$ if we let $C=\left(c_1 ... c_{p-1} \right)^T$. Note that $V$ has full column rank which implies $C=V^+x$ is a unique solution for $C$, where $V^+$ denotes the Penrose-inverse. This implies that $ \n C \n_\infty\le \n V^+\n_\infty \n A_{p,.}\n_\infty$. Similarly we get that $ \n C(n) \n_\infty\le \n V(n)^+\n_\infty \n A_{p,.}(n)\n_\infty$ where $V(n)^+$ is the Penrose inverse of the matrix $V(n)=\left(A_{1,.}(n)^T A_{2,.}(n)^T ... A_{p-1,.}(n)^T\right)$ and $C(n)=\left(c_1(n) ... c_{p-1}(n)\right)$. Since all the matrices $\{V(n)\}_n$ have rank $p-1$ and $\n V(n)-V \n_\infty$ by the law of large numbers, it follows that $\n V(n)^+- V^+\n_\infty\xrightarrow{a.s.}0$ which implies $\n V(n)^+\n_\infty \xrightarrow{a.s.}\n V^+\n_\infty$. By the law of large numbers we also have that $\n A_{p,.}(n) \n_\infty\xrightarrow{a.s.}\n A_{p,.}\n_\infty$. All of this implies that the sequences $\{c_u(n)\}_n$ are bounded for $u=1,...,p-1$. Suppose now that for some $1\le m< p$, $c_m(n)\not\to c_m$. Then since the sequence $\{c_m(n)\}_n$ is bounded there exists a subsequence $\{c_m(n_k)\}_k$ such that $c_m(n_k)\to c'_m\not=c_u$ and $c_v(n_k)\to c'_v$, $1\le v\le p-1$. Since $A_{v,.}(n)\to A_{v,.}$, $1\le v\le p$ it follows that $A_{p,.}=\sum_{u=1}^{p-1}c'_uA_{u,.}$ but $c'_u\not=c_u$ this contradicts the uniqueness of the decomposition and this gives us the result.
\end{proof}

We are now ready for the proof of Theorem \ref{optimalregthm}.
\begin{proof}[Proof of Theorem \ref{optimalregthm}]
	For the sake of brevity we will often times omit the dependence on $\textbf{n}$ for most of the estimators in this proof. First we will show that $\mathcal{B}$ is non-empty. Clearly the infinum along such an intersect will always exist but it might be a limit and not an attained minimum. We henceforth consider the infinum along the intersection of $f_i(\beta)$ and $f_j(\beta)$. The determinant of $G^u$ is a polynomial on $\Theta$, so by Lemma \ref{PolyLemma} (since it is clearly not the zero polynomial over $\Theta$) it is only zero on the set $N_1$ of Lebesgue measure zero in $\Theta$. But by definition $G^u$ is positive semi-definite so if it is of full rank then it is in fact positive definite. Expanding the risk function we see that
	$$R_{u}(\beta)=\E\left[Y_u^2\right]-2\beta\E\left[Y_uX_u\right]+\beta G^u\beta^T. $$
	Since $G^u$ is positive definite outside a set of measure zero in $\Theta$ we see that by a diagonalization of $G^u$, $G^u=ODO^T$, where $O$ is an orthogonal matrix and $D$ a diagonal matrix with positive eigenvalues $\lambda_1,...,\lambda_p$ such that $\n \beta O \n=\n\beta\n$. Letting $\tilde{\beta}=\beta O$ (recall that we defined $\beta$ as a row vector and note that $\tilde{\beta}O^T=\beta $), we have that 
	$$\beta G^u\beta^T=\beta ODO^T\beta^T=\tilde{\beta}D\tilde{\beta}^T=\sum_{v=1}^p\lambda_v\tilde{\beta}_v^2.$$
	So if we let $\tilde{\lambda}=\min_{1\le v\le p}\lambda_v$ then
	\begin{align*}
		R_{u}(\beta)
		&=
		\E\left[Y_u^2\right]-2\beta\E\left[Y_uX_u\right]+\beta G^u\beta^T
		=
		\E\left[Y_u^2\right]-2\tilde{\beta}O^T\E\left[Y_uX_u\right]+\sum_{u=1}^p\lambda_u\tilde{\beta}_u^2
		\\
		&=
		\E\left[Y_u^2\right]-2\langle\tilde{\beta}O^T,\E\left[Y_uX_u\right]\rangle+\sum_{v=1}^p\lambda_v\tilde{\beta}_v^2
		\ge 
		\E\left[Y_u^2\right]-2\n \tilde{\beta}O^T\n_2\n \E\left[ Y_uX_u\right]\n_2+\tilde{\lambda}\n \tilde{\beta}\n_2^2
		\\
		&\ge
		\E\left[Y_u^2\right]-2\n \beta\n_2 \E\left[\n Y_uX_u\n_2\right]+\tilde{\lambda}\n \tilde{\beta}\n_2^2
		=\E\left[Y_u^2\right] + \n\beta\n_2\left(\tilde{\lambda}\n\beta\n_2-2\E\left[|Y_u| \n X_u\n_2\right]\right)
		\\
		&\ge
		\E\left[Y_u^2\right] + \n\beta\n_2\left(\tilde{\lambda}\n\beta\n_2-2 \n Y_u\n_{L^2}\sum_{v=1}^p\n X_u(v)\n_{L^2}\right),
	\end{align*}
	where we used both the Jensen and the Cauchy-Schwarz inequality. Therefore if $w_i$ are non-negative (with at least one positive weight)
	$$f_i(\beta)\ge
	\kappa_i+\sum_{u=1}^kw_i(u)\left(\E\left[Y_u^2\right] + \n\beta\n_2\left(\tilde{\lambda}\n\beta\n_2-2 \n Y_i\n_{L^2}\sum_{v=1}^p\n X_i(v)\n_{L^2}\right)\right),$$ 
	outside a set of measure zero in $\Theta^k$. This implies that $\liminf_{\n \beta \n\to\infty}f_i(\beta)=\infty$ when the weights are non-negative and since at least this is the case for at least one $1\le i\le m$ and we take the maximum over the $f_i$s we can conclude that $\liminf_{\n \beta \n\to\infty}f(\beta)=\infty$. Let $\{\beta_n\}_n$ be a sequence such that $\lim_{n\to\infty}f(\beta_n)=\inf_{\beta\in\R^p}f(\beta)$, then we must have $\limsup_{n\to\infty}\n\beta_n\n<\infty$ (otherwise we would have $\limsup_{n\to\infty}f(\beta_n)=\infty$). Since the sequence $\{\beta_n\}_n$  is bounded, there is a subsequence  $\{\beta_{n_k}\}_{k}$ such that $\beta_{n_k}\to\beta^*\in\R^p$ and since $f$ is continuous $f(\beta_{n_k})\to f(\beta^*)=\inf_{\beta\in\R^p}f(\beta)$, so $\mathcal{B}$ is non-empty. In the special case \eqref{classic}, since we can assume $G^1,...,G^k$ all are positive definite, $f$ will be strictly convex in this case (the maximum of strictly convex functions is strictly convex). The statement now follows from the fact that a strictly convex function on an open set has a unique minimum ($\R^p$ in our case). Note that the strict convexity of $f$ in this case implies that for sufficiently large $\textbf{n}$, $\hat{f}$ will also be strictly convex which implies that $|\hat{B}(\textbf{n})|=1$, for such $\textbf{n}$. This will imply $\hat{B}(\textbf{n})\longrightarrow\mathcal{B}$ once we establish $\hat{B}(\textbf{n})\uparrow\mathcal{B}$ (in the general case). 
	\\
	\\
	Next we show that if $f_i$ and $f_j$ intersect then along their intersection there must be an argmin (with respect to $f_i$ and $f_j$, not necessarily with respect to $f$). There exists a sequence $\{\beta_l\}_l$ with $\n \beta_l \n\le C$ for some $C<\infty$ such that $\liminf_{\beta\in\R^p:f_i(\beta)=f_j(\beta)} f_i(\beta)=\lim_{l\to\infty}f_i(\beta_l)$ (otherwise we would get a contradiction similarly to before), so in other words
	$$\liminf_{\beta\in\R^p:f_i(\beta)=f_j(\beta)} f_i(\beta)=\liminf_{\beta\in\R^p:f_i(\beta)=f_j(\beta), \n \beta\n\le C} f_i(\beta).$$
	The set $\left\{\beta\in\R^p:f_i(\beta)=f_j(\beta)\right\}$ is closed since it is the kernel of the continuous function $f_i-f_j$. Therefore the set $\left\{\beta\in\R^p:f_i(\beta)=f_j(\beta), \n \beta\n\le C)\right\}$ is compact, but the infinum of a continuous function over a compact set is a minimum and hence all solutions in this case must be $\arg\min$ solutions. 
	\\
	\\
	We now consider the corresponding empirical case, i.e. we assume that the two estimators $\hat{f}_i$ and $\hat{f}_j$ intersect. By continuity and the law of large numbers it follows that for large $\textbf{n}$, the matrix $\hat{G}_i$ is also positive definite and we may consider the corresponding diagonalization of $\hat{G}_i=\hat{O}\hat{D}\hat{O}^T$. If we let $\hat{\tilde{\lambda}}$ denote the smallest eigenvalue of $\hat{D}$. We can than show analogously to the population case that
	$$\hat{R}_i(\beta)\ge \frac{1}{n} \sum_{l=1}^{n_i}Y_{i,l}^2+\n\beta\n_2\left(\hat{\tilde{\lambda}}\n\beta\n_2-2\sqrt{\frac{1}{n} \sum_{l=1}^{n_i}|Y_{i,l} |^2}\sqrt{\frac{1}{n} \sum_{l=1}^{n_i}\n X_{i,l}\n_2^2}\right).$$
	Therefore we may draw the same conclusion as before, namely if $\hat{f}_i$ and $\hat{f}_j$ intersect for large $\textbf{n}$ then we have an $\arg\min$ along this intersection. 
	\\
	\\
	We now move on to study the inflexion points of individual $f_i$s. To find an explicit solution to the inflexion point $\beta'=\arg\min_{\beta\in\R^p}f_i(\beta)$ we solve the equation $\nabla_\beta  f_i(\beta)=0$. Differentiating we see that,
	\begin{align*}
		&\nabla_\beta f_i(\beta)=2\sum_{u=1}^kw_i(u)\E\left[(X_i)^T\left(X_i\beta-Y^{A_i} \right)\right].
	\end{align*}
	Since $G^u$, $u=1,...,k$ are positive definite and due to the assumption of the weights, $\sum_{u=1}^kw_i(u)G^u$ is either positive definite or negative definite, either way it is full rank and we may therefore obtain, $\beta'=(\sum_{u=1}^kw_i(u)G^u)^{-1}\sum_{u=1}^kw_i(u)Z^u$ (which exists outside $N_1$, where $G^u$ is full rank). Furthermore, since we assumed that $G^i$ is of full rank it follows from the law of large numbers that $(\sum_{u=1}^kw_i(u)\hat{G}^u(\textbf{n}))^{-1}\sum_{u=1}^kw_i(u)\mathbb{Z}^u(\textbf{n})\xrightarrow{a.s.}\beta'$ (note that $\hat{G}^u(\textbf{n})$ must be of full rank for large $\textbf{n}$, by continuity of the determinant). This shows that plug-in estimation of the inflexion point of each $f_i$ is consistent.
	\\
	\\
	As we know, the minimum of $f$ can only be achieved at either an inflexion point for some $f_i$ or along the intersection (where $f$ might not be differentiable) of two $f_i$'s. Let $B_{inf}=\{\beta_{inf(1)},...,\beta_{inf(k)}\}$ denote the set of inflexion points corresponding to $f_1,...,f_k$ and let $\hat{B}_{inf}(\textbf{n})$ denote the corresponding estimators. If $f_i$ and $f_j$ intersect let $B^{i,j}$ be its set of argmin points along this intersection, this set will be finite outside a zero set in $\Theta^k$ (we will show this later). If $f_i$ and $f_j$ do not intersect then let $B^{i,j}=\emptyset$. Correspondingly, if $\hat{f}_i$ and $\hat{f}_j$ intersect we let $\hat{B}^{i,j}(\textbf{n})$ denote the set of argmin points along this intersection. If $\hat{f}_i$ and $\hat{f}_j$ do not intersect, we set $\hat{B}^{i,j}(\textbf{n})=\emptyset$. Let $B_{int}=B^{1,1}\cup...\cup B^{1,k}\cup B^{2,3}\cup...\cup B^{2,k}\cup...\cup B^{k-1,k}\cup B^{k,k}$ denote all the intersection points between the different $f_i$'s and let $\hat{B}_{int}(\textbf{n})$ denote the corresponding set of estimators. We also let $B=B_{int}\cup B_{inf}$ and $\hat{B}(\textbf{n})=\hat{B}_{int}(\textbf{n})\cup \hat{B}_{inf}(\textbf{n})$. For $\beta\in B^{i,j}$ to be a solution point in $\mathcal{B}$, we must have that $f_i(\beta)=f(\beta)$. Define $\hat{f}(\beta)=\bigvee_{i=1}^k\hat{f}_i(\beta)$. The argmin of $f$ must be achieved at a point in either $\hat{B}_{inf}(\textbf{n})$ or $\hat{B}_{int}(\textbf{n})$, i.e.
	\begin{align}
		&\mathcal{B}=\arg\min_{\beta\in \left(B_{inf} \cup  B_{int}\right)\cap\{\beta\in \R:\hspace{1mm} \exists 1\le i \le k,\hspace{1mm} f(\beta)=f_i(\beta)\}}f(\beta),
	\end{align}
	while
	\begin{align}\label{betaeqemp}
		&\hat{B}(\textbf{n})=\arg\min_{\beta\in \left(\hat{B}_{inf}(\textbf{n}) \cup \hat{B}_{int}(\textbf{n})\right)\cap\{\beta\in\R:\hspace{1mm} \exists 1\le i \le k, \hat{f}(\beta)= \hat{f}_i(\beta)\}} \hat{f}(\beta) .
	\end{align}
	We now construct the vector $V$ as follows, the first $k$ entries are $\beta_{inf(1)},...,\beta_{inf(k)}$ in chronological order. Then we place all the elements of $B^{i,j}$'s in rising order beginning with $B^{0,1}$ and ending with $B^{k-1,k}$ with the convention that we always place elements of $B^{i,j}$ for $i<j$ but not $i\ge j$ (if $B^{i,j}=\emptyset$ then no elements are placed). The individual elements of $B^{i,j}$ are to be ordered lexicographically. This way we can associate each argmin- candidate point with a unique index number in $V$. The number of elements in $V$ will be denoted $\mathcal{M}$. We then construct $\hat{V}_n$ completely analogously from the corresponding estimators and let $\hat{\mathcal{M}}(\textbf{n})$ be the number of elements in $\hat{V}_\textbf{n}$. Let us define the optimal choice indices
	$$L=\arg\min_{1\le l \le M}f(V(l)), $$
	(note that $L$ is a set in general) so that $V(L)=\mathcal{B}$ and similarly let 
	$$\hat{L}_\textbf{n}=\arg\min_{1\le l \le \hat{\mathcal{M}}(\textbf{n})}\hat{f}(\hat{V}_\textbf{n}(l)), $$
	so that $\hat{B}(\textbf{n})= \hat{V}_\textbf{n}(\hat{L}_\textbf{n})$. Let us also define 
	$$\hat{L}_{\textbf{n},\epsilon'}=\{l\not\in \hat{L}_\textbf{n}: \hat{f}(\hat{V}_\textbf{n}(l))-\hat{f}(\hat{V}_\textbf{n}(l'))<\epsilon', l'\in \hat{L}_\textbf{n}\},$$ 
	the set of candidates that when evaluated are within an $\epsilon'$-distance of the chosen argmin candidates. We will show that for large enough $\textbf{n}$, $\mathcal{M}(\textbf{n})=\mathcal{M}$, $|\hat{B}^{i,j}(\textbf{n})|=|B^{i,j}|$ (for all $i,j$) and that $\hat{V}_\textbf{n}(l)\xrightarrow{a.s.}V(l)$ for $1\le l\le \mathcal{M}$. For all $l\not\in L$ there exists $a>0$ such that $V(l)-V(u)>a$, for all $u\in L$. Therefore we can set $\epsilon=a$ as stated in the Theorem. Since $\hat{V}_\textbf{n}(l)\xrightarrow{a.s.}V(l)$ for $1\le l\le \mathcal{M}$ and for large enough $\textbf{n}$, $\mathcal{M}(\textbf{n})=\mathcal{M}$ it follows that if $l\not\in L$, $\hat{V}_\textbf{n}(l)-\hat{V}_\textbf{n}(u)>a$, for large enough $\textbf{n}$, $u\in L$. This implies that $\hat{L}_\textbf{n}\subseteq L$ for large enough $\textbf{n}$, i.e. we always make a valid choice among the candidates (although we might miss some candidates that are close to the plug-in argmin) and since all the candidates converge this proves that $\hat{B}(\textbf{n})\uparrow \mathcal{B}$. Meanwhile for large enough $\textbf{n}$, we also have $\hat{L}_{\textbf{n},\epsilon'}=L$ if $\epsilon'<\epsilon$ which shows that $\dist(\hat{B}_{\epsilon'}(\textbf{n}),\mathcal{B})\to 0$ a.s.. It now remains to show that $\mathcal{M}(\textbf{n})=\mathcal{M}$ for large $\textbf{n}$ and that $\hat{V}_\textbf{n}(l)\xrightarrow{a.s.}V(l)$ for $1\le l\le \mathcal{M}$. For the first statement we already know that all the inflexion points for each $f_i$ will converge so it remains to show that all candidate points from intersections as well as the second statement. We now proceed with showing this, i.e. we must show that with a lexicographical ordering of $B^{i,j}$, $\beta^{i,j}_1<...<\beta^{i,j}_q$ and of $\hat{B}^{i,j}(\textbf{n})$, $\hat{\beta}_1^{i,j}<...<\hat{\beta}_{\hat{q}}^{i,j}$ (with $\hat{q}=q$ for large $\textbf{n}$) we have that $\hat{\beta}^{i,j}_u\xrightarrow{a.s.}\beta^{i,j}_u$ for $u=1,...q$. 
	\\
	\\
	Since each $f_i$ is either convex or concave, by the the Kuhn-Tucker Theorem, a necessary condition for $\beta^*$ to be a minimum point for $f_{i}(\beta)$ subject to $g_{i,j}(\beta)=0$ is that the corresponding Lagrangian, $\mathcal{L}$ fulfils, $\nabla_\beta \mathcal{L}(\beta^*,\lambda^*)=0$ for some $\lambda^*$ and $g_{i,j}(\beta^*)=0$. This will only have a finite set of solutions in terms of $\lambda^*$ (as it will be polynomial in $\lambda^*$) and correspondingly a finite set of $\beta$'s, we then choose whichever one that minimizes $f_{i}$. In case $f_{i}$ is convex (i.e. non-negative weights), the Lagrangian is given by (if it is concave, i.e. non-positive weights, we simply change the sign in front of $\lambda$)
	
	\begin{align*}
		\mathcal{L}(\beta,\lambda)&=f_{i}(\beta)-\lambda g_{i,j}(\beta)
		=
		\sum_{l=0}^p\sum_{u=1}^k\beta_l^2\left(\E\left[w_i(u)X_u(l)^2-\lambda\left( (w_i(u)-w_j(u))X_u(l)^2\right)\right] \right)
		\\
		&-2\sum_{l=1}^p\sum_{u=1}^k\beta_l\E\left[w_i(u)Y_u(X_u(l))-\lambda\left((w_i(u)-w_j(u))Y_i(X_u(l))\right)\right]
		\\
		&+2\sum_{l_1=0}^p\sum_{l_2\not=l_1}^p\sum_{u=1}^k\beta_{l_1}\beta_{l_2}\E\left[w_i(u)X_u(l_1)X_u(l_2)-\lambda\left((w_i(u)-w_j(u))X_u(l_1)X_u(l_2)\right)\right]
		\\
		&+
		\sum_{u=1}^k\E\left[w_i(u)Y_u^2-\lambda \left((w_i(u)-w_j(u))Y_u^2\right)\right]+\kappa_i-\lambda(\kappa_i-\kappa_j)
		\\
		&=\sum_{l=1}^p\beta_l^2\left(a_i(l)-\lambda\left( a_i(l)-a_j(l)\right) \right)
		-2\sum_{l=1}^p\beta_l\left(b_i(l)-\lambda\left(b_i(l)-b_j(l)\right)\right)
		\\
		&+2\sum_{l_1=0}^p\sum_{l_2\not=l_1}^p\beta_{l_1}\beta_{l_2}\left(a_i(l_1,l_2)-\lambda\left(a_i(l_1,l_2)-a_j(l_1,l_2)\right)\right)
		+c_i-\lambda\left(c_i-c_j\right),
	\end{align*}
	where 
	\begin{itemize}
		\item $a_i(l_1,l_2)=\sum_{u=1}^kw_i(u)\E\left[X_u(l_1)X_u(l_2)\right]$,
		\item $a_i(l)=\sum_{u=1}^kw_i(u)\E\left[X_u(l)^2\right]$,
		\item $b_i(v)=\sum_{u=1}^kw_i(u)\E\left[X_u(v)Y_u\right]$ and
		\item $c_i=\kappa_i+\sum_{u=1}^kw_i(u)\E\left[Y_u^2\right]$.
	\end{itemize}
	Setting $\nabla_\beta\mathcal{L}(\beta,\lambda)=0$ yields $M^{i,j}(\lambda)\beta=C^{i,j}_{\lambda}$, in other words
	\begin{align*}
		&\beta_l\left(\E\left[(X_i(l))^2-\lambda\left( (X_i(l))^2-(X_j(l))^2\right)\right] \right)
		=
		\E\left[Y_i(X_i(l))-\lambda\left(Y_i(X_i(l))-Y^{A_j}(X_j(l))\right)\right]
		-
		\\
		&\sum_{l_2\not=l}^p\beta_{l_2}\E\left[X_i(l)X_i(l_2)-\lambda\left(X_i(l)X_i(l_2)-X_j(l)X_j(l_2)\right)\right],
	\end{align*}
	where we define the $p\times p$ matrix $M^{i,j}(\lambda)$ and the vector $C^{i,j}(\lambda)$ as in Definition \ref{parammatrix} (and substitute in our definitions of $a_i,c_i$ and $b_i$) i.e. if $i$ is such that $f_i$ has non-negative weights,
	$$M^{i,j}(\lambda)_{l,l} =a_i(l)-\lambda\left(a_i(l)-a_j(l)\right)$$
	when $u\not=v$
	$$M^{i,j}(\lambda)_{u,v} =a_i(u,v)-\lambda\left(a_i(u,v)-a_j(u,v)\right)$$
	and
	$$C^{i,j}(\lambda)_u=b_i(u)-\lambda\left(b_i(u)-b_j(v)\right).$$
	While if $f_i$ has non-positive weights we change the sign in front of $\lambda$.
	Let $P(\lambda)=\det\left(M^{i,j}(\lambda)\right)$. We will now establish the fact that outside a set of measure zero in $\Theta^k$, $M^{i,j}(\lambda)\beta=C^{i,j}(\lambda)$ has no solutions for any $\lambda\in\R$ such that $P(\lambda)=0$. By Lemma \ref{matrixLemma}, outside a set of measure zero $N_2\in\Theta^k$, $\mathrm{rank}(M^{i,j}(\lambda))=p-1$ for $\lambda\in\Lambda$ and 
	\begin{align}\label{linearcomb}
		&M^{i,j}_{p,.}(\lambda)=\sum_{u=1}^{p-1}s_u(\theta^{(k)},\lambda)M_{u,.}(\lambda)
	\end{align}
	with all $s_u(\theta^{(k)},\lambda)\not=0$. From Lemma \ref{matrixNoTiesLemma}, we see that outside a zero set $N_3\in\Theta^k$, $\{s_u(\theta^{(k)},\lambda)\}_{u=1}^{p-1}$ are all rational functions on $\R\times\Theta^k$. Applying the same row operations to the vector $C^{i,j}(\lambda)$ as $M^{i,j}(\lambda)$, we now have that $M^{i,j}(\lambda)\beta=C^{i,j}(\lambda)$ has no solutions if
	$$\left(C^{i,j}_{\lambda}(p)-\sum_{v\not=p}s_v(\theta,\lambda) C^{i,j}_{\lambda}(v)\right)\not=0.$$
	Let $G$ denote the lowest common denominator of the terms in $C^{i,j}_{\lambda}(p)-\sum_{v\not=p}s_v(\theta^{(k)},\lambda) C^{i,j}_{\lambda}(v)$ then 
	$$Q(\theta^{(k)},\lambda)= G(\theta^{(k)},\lambda)\left(C^{i,j}_{\lambda}(p)-\sum_{v\not=p}s_v(\theta^{(k)},\lambda) C^{i,j}_{\lambda}(v)\right)$$ 
	is a polynomial on $\R\times\Theta^k$. To show that $Q$ has no roots in $\Lambda$ is equivalent to showing that $P$ and $Q$ have no common roots, which is in turn equivalent to showing that their resultant is not zero. Again, we need to show that there exists $\theta^{(k)}\in\Theta^k$ such that $P$ and $Q$ have no common roots. Take the matrix $\tilde{M}$ from the proof of Lemma \ref{matrixNoTiesLemma}. As we already have seen for $\tilde{M}$, $s_1\equiv...\equiv s_{p-1}\equiv 1$. We readily see that in this case $G\equiv 1$ (since $G$ is the lowest common denominator of $\{s_u\}_{u=1}^{p-1}$), so any choice of parameters that leads to the first degree polynomial $C^{i,j}_{\lambda}(p)-\sum_{v\not=p}s_v(\theta^{(k)},\lambda) C^{i,j}_{\lambda}(v)$ not having a root in $\lambda=0$, shows that outside a zero set $N_4$ there are no solutions to $M^{i,j}(\lambda)\beta=C^{i,j}(\lambda)$ for any $\lambda\in\Lambda$. We continue with the case when $\lambda\in\R$ is such that $M^{i,j}(\lambda)$ is full rank. If $M^{i,j}(\lambda)$ is full-rank then $\beta(\lambda)=(M^{i,j})^{-1}(\lambda)C^{i,j}_{\lambda}$, to find $\lambda$ we solve $g_{i,j}(\beta(\lambda))=0$. Note that 
	$$\beta(\lambda)_u=\frac{1}{\det\left(M^{i,j}(\lambda)\right)}\sum_{v=1}^p\det\left(M^{i,j}_{u,v}(\lambda)\right)C^{i,j}(\lambda)(v),$$ 
	where we again use the convention that $M^{i,j}_{u,v}$ is the matrix $M^{i,j}$ with row $u$ and column $v$ removed. So for $\lambda\not\in\Lambda$
	\small
	\begin{align*}
		&g_{i,j}(\beta(\lambda))=\left(a_i(l)-a_j(l)\right)\sum_{l=1}^p\left(\frac{1}{\det\left(M^{i,j}(\lambda)\right)}\sum_{v=1}^p\det\left(M^{i,j}_{l,v}(\lambda)\right)(-1)^{l+v}C^{i,j}(\lambda)(v)\right)^2+c_i-c_j
		\\
		&-2\left(b_i(l)-b_j(l)\right)\sum_{l=1}^p\left(\frac{1}{\det\left(M^{i,j}(\lambda)\right)}\sum_{v=1}^p\det\left(M^{i,j}_{l,v}(\lambda)\right)(-1)^{l+v}C^{i,j}(\lambda)(v)\right)
		\\
		&+2\left(a_i(l_1,l_2)-a_j(l_1,l_2)\right)\sum_{l_1=1}^p\sum_{l_2\not=l_2}\left(\frac{1}{\det\left(M^{i,j}(\lambda)\right)}\sum_{v=1}^p\det\left(M^{i,j}_{l_1,v}(\lambda)\right)(-1)^{l_1+v}C^{i,j}(\lambda)(v)\right)
		\\
		&\times\left(\frac{1}{\det\left(M^{i,j}(\lambda)\right)}\sum_{v=1}^p\det\left(M_{l_2,v}(\lambda)\right)(-1)^{l_2+v}C^{i,j}(\lambda)(v)\right).
	\end{align*}
	\normalsize
	Furthermore from the above expression we see that for $\lambda\not\in\Lambda$, $g_{i,j}(\beta(\lambda))$ has the same roots as 
	$$\tilde{P}(\lambda)=\det\left(M^{i,j}(\lambda)\right)^2g_{i,j}(\beta(\lambda)),$$ 
	which is a polynomial. Outside a set of measure zero $N_5\in\Theta^k$, $\tilde{P}(\lambda)$ has only simple roots, according to Lemma \ref{discriminantLemma}. Let us now define $N=N_1\cup...\cup N_5$. Letting $\mathcal{R}$ denote the (finite) set of roots to $g_{i,j}(\beta(\lambda))=0$ then the set of intersection $\arg\min$'s will now be given by 
	$$B^{i,j}=\left\{M^{-1}(\lambda^{i,j})C_{\lambda^{i,j}}: \lambda^{i,j}\in\arg\min_{\lambda\in\mathcal{R}}f_i\left(M^{-1}(\lambda)C^{i,j}(\lambda)\right)\right\}.$$
	From here on out we fix $\theta^{(k)}\in\Theta^k\setminus N$. We now consider the plug-in estimator of the minima along the intersection between $f_i$ and $f_j$,
	$$\arg\min_{\beta\in\R^p} \hat{f}_i(\beta),\texttt{ subject to } \hat{f}_i(\beta)=\hat{f}_j(\beta).$$
	We also consider the plug-in Lagrangian
	\begin{align*}
		&\left(\hat{\mathcal{L}}(\beta,\lambda)\right)(\textbf{n})=\left(\hat{f}_i(\beta)\right)(\textbf{n})-\lambda \hat{g}_{i,j}(\beta)
		=
		\sum_{l=1}^p\beta_l^2\left(-\lambda\left( \hat{a}_i(l)-\hat{a}_j(l)\right) \right)
		\\
		&-2\sum_{l=1}^p\beta_l\left(\hat{b}_i(l)-\lambda\left(\hat{b}_i(l)-\hat{b}_j(l)\right)\right)
		+2\sum_{l_1=0}^p\sum_{l_2\not=l_1}^p\beta_{l_1}\beta_{l_2}\left(\hat{a}_i(l_1,l_2)-\lambda\left(\hat{a}_i(l_1,l_2)-\hat{a}_j(l_1,l_2)\right)\right)
		\\
		&+\hat{c}_i-\lambda\left(\hat{c}_i-\hat{c}_j\right),
	\end{align*}
	where 
	\begin{itemize}
		\item $\hat{g}_{i,j}(\beta)=\hat{f}_i(\beta)-\hat{f}_j(\beta)$,
		\item $\hat{a}_i(l)=\sum_{u=1}^kw_i(u)\frac{1}{n_u}\sum_{v=1}^{n_i}X_{u,v}(l)^2$,
		\item $\hat{a}_i(l_1,l_2)=\sum_{u=1}^kw_i(u)\frac{1}{n_u}\sum_{v=1}^{n_u}X_{u,v}(l_1)X_{u,v}(l_2)$,
		\item $\hat{b}_i(l)=\sum_{u=1}^kw_i(u)\frac{1}{n_u}\sum_{v=1}^{n_i}X_{u,v}(l)Y_{u,v}$, and
		\item $\hat{c}_i=\kappa_i+\sum_{u=1}^kw_i(u)\frac{1}{n_u}\sum_{v=1}^{n_u}Y_{u,v}^2$.
	\end{itemize}
	Then similarly to before we see that solving $\nabla_\beta \left(\hat{\mathcal{L}}(\beta,\lambda)\right)(\textbf{n})=0$ leads to $\widehat{M}^{i,j}(\lambda)\beta=\widehat{C}^{i,j}_\lambda$, where 
	$$\widehat{C}^{i,j}_u(\lambda)=\hat{b}_i(u)-\lambda\left(\hat{b}_i(u)-\hat{b}_j(v)\right), $$
	$$\widehat{M}^{i,j}_{l,l}(\lambda)=\hat{a}_i(l)-\lambda\left(\hat{a}_i(l)-\hat{a}_j(l)\right),$$
	and when $u\not=v$
	$$\widehat{M}^{i,j}(\lambda)_{u,v} =-\left(\hat{a}_i(u,v)-\lambda\left(\hat{a}_i(u,v)-\hat{a}_j(u,v)\right)\right).$$
	Let $\hat{\Lambda}_n$ denote the set of roots (in terms of $\lambda)$ to $\hat{P}(\lambda)= \det\left(\hat{M}^{i,j}(\lambda)\right)=0$ and let $\hat{\Lambda}_\R(\textbf{n})$ denote the real roots of this polynomial. We wish to establish that for large enough $\textbf{n}$, $\hat{M}^{i,j}(\lambda)\beta=\hat{C}^{i,j}_\lambda$ has no solutions for $\lambda\in\hat{\Lambda}_n$. Recall that $P$ only has simple roots which implies, by Lemma \ref{simpleroots} and the law of large numbers that the elements in $\hat{\Lambda}_\R(\textbf{n})$ converge a.s. to those in $\Lambda_\R$, where we let $\Lambda_\R=\{\lambda_1,...,\lambda_m\}$ (where $\lambda_1<...<\lambda_m$ and $m$ depends on $\theta^{(k)}$ but $\theta^{(k)}$ is fixed at this point) denote the real roots of $P$. Note that by the proof of Lemma \ref{matrixLemma}, for each $\lambda\in\Lambda$ (and therefore also in $\Lambda_\R$) the first $p-1$ first minors along the first column in $M^{i,j}(\lambda)$ is bounded below by some $\delta>0$ ($\delta>0$ also depends on $\theta^{(k)}$, but $\theta^{(k)}$ is fixed), i.e. $\left|\det\left(M_{u,1}(\lambda)\right)\right|>\delta$ for each $\lambda\in\Lambda$ and $1\le u\le p-1$. For large enough $\textbf{n}$, the number of elements in $\hat{\Lambda}_n$ and $\Lambda$ coincides by Lemma \ref{simpleroots}. From Lemma \ref{simpleroots} we also have $\hat{\lambda}_z\xrightarrow{a.s.}\lambda_z$, for $1\le z\le m$. Since the determinant is continuous it follows that $\left|\det\left(\hat{M}^{i,j}_{1,v}(\hat{\lambda}_z)\right)\right|>\delta$, $v=1,...,p-1$ for large enough $\textbf{n}$. Therefore $\mathsf{rank}\left( \hat{M}^{i,j}(\hat{\lambda}_z)\right)=p-1$ and we have the unique representation (in terms of the rows $\hat{M}^{i,j}_{u,.}(\lambda)$)
	\begin{align}\label{linearcombhat}
		&\hat{M}^{i,j}_{p,.}(\hat{\lambda}_z)=\sum_{u=1}^{p-1}\hat{s}_u(\hat{\lambda}_z)\hat{M}^{i,j}_{u,.}(\hat{\lambda}_z)
	\end{align}
	for all $\hat{\lambda}_z\in\hat{\Lambda}(\textbf{n})_\R$ with all $\hat{s}_u(\hat{\lambda}_z)\not=0$. Note that 
	\begin{align}\label{Mdiff}
		&\n \hat{M}^{i,j}(\hat{\lambda}_z)-M(\lambda_z) \n\le \n \hat{M}^{i,j}(\hat{\lambda}_z-\lambda_z) \n+\n\hat{M}^{i,j}(\lambda_z)-M(\lambda_z) \n,
	\end{align}
	for the first term above $\n \hat{M}^{i,j}(\hat{\lambda}_z-\lambda_z) \n=\left|\hat{\lambda}_z-\lambda_z\right|\n\hat{M}'(\textbf{n})\n$, where (assuming $f_i$ has non-negative weights) $\hat{M}'_{u,v}(\textbf{n})=\hat{a}_i(u,v)-\hat{a}_j(u,v)$ ($\hat{M}'_{u,v}(\textbf{n})=\hat{a}_i(u,v)+\hat{a}_j(u,v)$ if $f_i$ has non-positive weights) when $u\not=v$ and $\hat{M}'_{l,l}(\textbf{n})=\hat{a}_i(l)-\hat{a}_j(l)$ ($\hat{M}'_{l,l}(\textbf{n})=\hat{a}_i(l)+\hat{a}_j(l)$ if $f_i$ has non-positive weights). By the law of large numbers $\hat{M}'(\textbf{n})$ converges to the matrix $M'$ with $M'_{u,v}=a_i(u,v)-a_j(u,v)$ when $u\not=v$ and $M'_{l,l}=a_i(l)-a_j(l)$. Due to the fact that $\hat{\lambda}_z\xrightarrow{a.s.}\lambda_z$ it is therefore clear that the first term in \eqref{Mdiff} vanishes. The second therm in \eqref{Mdiff} will vanish by the law of large numbers, hence $\n \hat{M}^{i,j}(\hat{\lambda}_z)-M^{i,j}(\lambda_z) \n\xrightarrow{a.s.}0$.
	Due to the fact that $\n \hat{M}^{i,j}(\hat{\lambda}_z)-M^{i,j}(\lambda_z) \n\xrightarrow{a.s.}0$ and Lemma \ref{coefficientslemma} it follows that $\hat{s}_u(\hat{\lambda}_z)\xrightarrow{a.s.}s_u(\lambda_z)$ for $u=1,...,p-1$ and $1\le z\le m$. Let
	$$H(\lambda)= G(\lambda)Q(\lambda)= \left(C^{i,j}_{\lambda}(p)-\sum_{v\not=p}s_v(\theta,\lambda) C^{i,j}_{\lambda}(v)\right),$$
	we know that since we chose $\theta^{(k)}\not\in N$, $|H(\lambda)|>\delta'$ (since $Q(\lambda)\not=0$) for some $\delta'>0$ and every $\lambda\in \Lambda(\theta^{(k)})$. We now let 
	$$\hat{H}(\lambda)= \left(\hat{C}^{i,j}_{\lambda}(p)-\sum_{v\not=p}\hat{s}_v(\lambda) \hat{C}^{i,j}_{\lambda}(v)\right),$$ 
	it follows that $\hat{H}(\hat{\lambda}_z)\xrightarrow{a.s.}H(\lambda_z)$ by the fact that $\hat{s}_u(\hat{\lambda}_z)\xrightarrow{a.s.}s_u(\lambda_z)$ and $\hat{C}^{i,j}_{\hat{\lambda}_z}(v)\xrightarrow{a.s.}C^{i,j}_{\lambda_z}(v)$ for $v=1,...p$ (this can be proved by employing the same exact strategy as for \eqref{Mdiff}). So for large $\textbf{n}$, $|\hat{H}(\hat{\lambda}_z)|>\delta$ which implies that there are no solutions to $\hat{M}^{i,j}(\lambda)\beta=\hat{C}^{i,j}_\lambda$ for $\lambda\in \hat{\Lambda}_\R(\textbf{n})$ for large enough $\textbf{n}$. Let $\hat{\beta}(\lambda)=(\hat{M}^{i,j})^{-1}(\lambda)\hat{C}^{i,j}_\lambda$ for $\lambda\not\in \hat{\Lambda}_n$. To find any candidates for $\hat{\lambda}$ we then solve $\hat{g}_{i,j}(\beta(\lambda))=0$. Similarly to before we see that
	\small
	\begin{align}\label{groots}
		\hat{g}_{i,j}(\beta(\lambda))&=\sum_{l=1}^p\left(\frac{1}{\det\left(\hat{M}^{i,j}(\lambda)\right)}\sum_{v=1}^p\det\left(\hat{M}^{i,j}_{l,v}(\lambda)\right)(-1)^{l+v}\hat{C}^{i,j}_\lambda(v)\right)^2\left(\hat{a}_i(l)-\hat{a}_j(l)\right)+\hat{c}_i-\hat{c}_j\nonumber
		\\
		&-2\sum_{l=1}^p\left(\frac{1}{\det\left(\hat{M}^{i,j}(\lambda)\right)}\sum_{v=1}^p\det\left(\hat{M}^{i,j}_{l,v}(\lambda)\right)(-1)^{l+v}\hat{C}^{i,j}_\lambda(v)\right)\left(\hat{b}_i(l)-\hat{b}_j(l)\right)\nonumber
		\\
		&+2\left(\hat{a}_i(l_1,l_2)-\hat{a}_j(l_1,l_2)\right)\sum_{l_1=1}^p\sum_{l_2\not=l_2}\left(\frac{1}{\det\left(\hat{M}^{i,j}(\lambda)\right)}\sum_{v=1}^p\det\left(\hat{M}^{i,j}_{l_1,v}(\lambda)\right)(-1)^{l_1+v}\hat{C}^{i,j}_\lambda(v)\right)
		\nonumber
		\\
		&\times\left(\frac{1}{\det\left(\hat{M}^{i,j}(\lambda)\right)}\sum_{v=1}^p\det\left(\hat{M}^{i,j}_{l_2,v}(\lambda)\right)(-1)^{l_2+v}\hat{C}^{i,j}_\lambda(v)\right).
	\end{align}
	\normalsize
	Analogously to the population case, if let $\hat{\tilde{P}}(\lambda)=\det\left(\hat{M}^{i,j}(\lambda)\right)^2\hat{g}_{i,j}(\beta(\lambda))$ then $\hat{\tilde{P}}(\lambda)$ has the same roots as $\hat{g}_{i,j}(\beta(\lambda))$. By the law of large numbers and continuity we see that $\det\left(\hat{M}^{i,j}(\lambda)\right)\xrightarrow{a.s.}\det\left( M^{i,j}(\lambda)\right)\not=0$ (since $\lambda\not\in\Lambda$), $\det\left(\hat{M}^{i,j}_{u,v}(\lambda)\right)\xrightarrow{a.s.} \det\left(M^{i,j}_{u,v}(\lambda)\right)$ for all $u,v$, $\hat{b}_.(.)\xrightarrow{a.s.}b_.(.)$, $\hat{a}_.(.)\xrightarrow{a.s.}a_.(.)$, $\hat{a}_.(.,.)\xrightarrow{a.s.}a_.(.,.)$ and finally that $\hat{c}_.\xrightarrow{a.s.}c_.$. We can thus conclude that all coefficients of $\hat{\tilde{P}}$ converge to the corresponding coefficients of $\tilde{P}$ and therefore, by Lemma \ref{simpleroots}, the real roots of $\hat{\tilde{P}}$ converge to those of $\tilde{P}$ and that for large enough $\textbf{n}$, the number of such roots of $\tilde{P}$ and $\hat{\tilde{P}}$ coincides. Letting $\mathcal{R}=\{\lambda_1,...,\lambda_m\}$ (with $\lambda_1<...<\lambda_{m}$) and $\hat{\mathcal{R}}=\{\hat{\lambda}_1,...,\hat{\lambda}_{\hat{m}}\}$ (with $\hat{\lambda}_1<...<\hat{\lambda}_{\hat{m}}$) denote the real roots (which are simple) of $\tilde{P}$ and $\hat{\tilde{P}}$ respectively we have that $\hat{\lambda}_u\xrightarrow{a.s.}\lambda_u$ for $u=1,..,m$ and for large $\textbf{n}$, $\hat{m}=m$, i.e. the elements of $\hat{\mathcal{R}}$ converge to those of $\mathcal{R}$. We now define
	$$\hat{B}^{i,j}(\textbf{n})=\left\{\hat{M}^{-1}(\lambda^{i,j})\hat{C}^{i,j}_{\lambda^{i,j}}: \lambda^{i,j}\in\arg\min_{\lambda\in\hat{\mathcal{R}} }\hat{f}_i\left((\hat{M}^{i,j})^{-1}(\lambda)\hat{C}^{i,j}_\lambda\right)\right\}. $$
	By the law of large numbers and continuity (since $\det\left(M^{i,j}(\lambda)\right)\not=0$ for $\lambda\in\Lambda$) 
	$$(\hat{M}^{i,j})^{-1}(\lambda)\hat{C}^{i,j}_{\lambda}\xrightarrow{a.s.}(M^{i,j})^{-1}(\lambda)C^{i,j}(\lambda),$$ 
	for $\lambda\not\in\Lambda$. Therefore $\max_{\beta\in\hat{B}^{i,j}}d\left(\beta,B^{i,j}\right)\xrightarrow{a.s.}0$ as well as $|\hat{B}^{i,j}(\textbf{n})|=|B^{i,j}|$ for large $\textbf{n}$.
	With a lexicographical ordering of $B^{i,j}$, $\beta^{i,j}_1<...<\beta^{i,j}_q$ and of $\hat{B}^{i,j}(\textbf{n})$, $\hat{\beta}_1^{i,j}<...<\hat{\beta}_{\hat{q}}^{i,j}$ (with $\hat{q}=q$ for large $\textbf{n}$) we have that $\hat{\beta}^{i,j}_u\xrightarrow{a.s.}\beta^{i,j}_u$ for $u=1,...q$, as we set out to prove.
\end{proof}

	\bibliographystyle{plainnat}
	\bibliography{bibliopaper}

\end{document}